\documentclass{amsart}
\usepackage[all]{xy}
\usepackage{amssymb,amsmath}

\theoremstyle{plain}
\newtheorem{thm}{Theorem}[section]

\newtheorem{prop}[thm]{Proposition}

\theoremstyle{definition}

\newtheorem{rmk}[thm]{Remark}

\def\O{\mathcal{O}}
 
\def\rk{\mathrm{rk}}

\begin{document}

\title{Excess dimension for secant loci in symmetric products of curves}
\author{Marian Aprodu}
\address{Faculty of Mathematics and Computer Science, University of Bucharest, 14 Academiei Street, 010014 Bucharest, Romania}
\address{Simion Stoilow Institute of Mathematics of the Romanian Academy, P.O. Box 1-764, 014700 Bucharest, Romania}
\author{Edoardo Sernesi}
\email{marian.aprodu@imar.ro, marian.aprodu@fmi.unibuc.ro}
\address{Dipartimento di Matematica e Fisica, Universit\`a degli Studi Roma Tre, Largo San Leonardo Murialdo, I-00146 Roma, Italy}
\email{sernesi@mat.uniroma3.it}
\thanks{We are grateful to the two referees for having made pertinent remarks on the first version of the manuscript. We thank the Max Planck Institute Math. Bonn for hospitality during the preparation of this work. Marian Aprodu was partly supported by the CNCS-UEFISCDI grant PN-II-PCE-2011-3-0288. }

\maketitle

\begin{abstract}
We extend a result of W. Fulton, J. Harris and R. Lazarsfeld \cite{Fulton-Harris-Lazarsfeld} to secant loci in symmetric products of curves. We compare three secant loci and prove the the dimensions of bigger loci can not be excessively larger than the dimension of smaller loci.
\end{abstract}

\section{Introduction}

In this Note we study the following problem that is connected to the geometry of secant loci and has already appeared in \cite{Fulton-Harris-Lazarsfeld}. Let $m$, $n$ and $k\le \mathrm{min}\{m,n\}$ be three positive integers, $X$ be an integral algebraic variety over $\mathbb C$, and consider a diagram of vector bundles:
\[
\xymatrix{
0 \ar[r] & H \ar[r] & F' \ar[r]^-\pi & F \ar[r] & 0 \\
&&E \ar[u]^-\sigma\ar[ur]_-{\pi\sigma}}
\]
where $\rk(E)=m$, \ $\rk(F')=n+1$ and $\rk(F)=n$. We propose to compare the dimensions of closed subschemes of $X$ defined as
\[
D_k(\sigma)=\{x\in X|\mathrm{rank}(\sigma_x)\le k\},
\]
\[
D_k(\pi\sigma)=\{x\in X|\mathrm{rank}(\pi_x\circ\sigma_x)\le k\}
\]
and 
\[
D_{k+1}(\sigma)=\{x\in X|\mathrm{rank}(\sigma_x)\le k+1\}.
\]

The comparison of $D_k(\sigma)$ and $D_k(\pi\sigma)$ supposing that $E^*\otimes H$ is ample is the content of  \cite[Lemma 4]{Fulton-Harris-Lazarsfeld}, however, we try to avoid here the ampleness hypothesis and replace it by weaker assumptions.

Obviously, $D_k(\sigma)\subset D_k(\pi\sigma)\subset D_{k+1}(\sigma)$ and in general the inclusions are strict. However, we shall prove that the dimension of bigger loci cannot increase too much compared to the dimension of smaller loci, similarly to \cite[Lemma 4]{Fulton-Harris-Lazarsfeld}.
To this end, we place ourselves first in the generic situation, section \ref{sec:generic}. This is one of the usual tricks used in the study of determinantal subschemes, see for example \cite{Arbarello-Cornalba-Griffiths-Harris}. In section \ref{sec:comparison} we find explicit comparison bounds, and, in the last part, we apply these bounds to the case of secant loci inside symmetric products of curves, which are defined as degeneracy loci of suitable sheaf morphisms. For canonical line bundles, the result specialises to the well-known "excess linear series" Theorem of Fulton-Harris-Lazarsfeld, \cite{Fulton-Harris-Lazarsfeld}. 


\section{The generic situation}
\label{sec:generic}

For any positive integers $a$, $b$ and $k\le \mathrm{min}\{a,b\}$ denote by $M(a,b)$ the variety of complex $a\times b$ matrices and by $M_k(a,b)\subset M(a,b)$ the subscheme of matrices of rank $\le k$. Its singular locus is precisely $M_{k-1}(a,b)$.  

\medskip

Given two integers $m$ and $n$ consider the morphism
\[
\mu:M(m,n+1)\times M(n+1,n)\to M(m,n)
\]
given by matrix multiplication.

\medskip

We compare the intersection of the three closed  subsets 
\[
\mu^{-1}(M_k(m,n)),\
M_k(m,n+1)\times M(n+1,n)
\mbox{ and }M_{k+1}(m,n+1)\times M(n+1,n)
\] 
of $M(m,n+1)\times M(n+1,n)$ with the complement of $M(m,n+1)\times M_{n-1}(n+1,n)$.

\medskip

Clearly  
$(M_k(m,n+1)\times M(n+1,n))\setminus (M(m,n+1)\times M_{n-1}(n+1,n))$ is contained in $\mu^{-1}(M_k(m,n))\setminus (M(m,n+1)\times M_{n-1}(n+1,n))$
and
$\mu^{-1}(M_k(m,n))\setminus (M(m,n+1)\times M_{n-1}(n+1,n))$ is contained in $(M_{k+1}(m,n+1)\times M(n+1,n))\setminus (M(m,n+1)\times M_{n-1}(n+1,n))$.

\medskip

We compute:
\begin{align}
\dim (M_k(m,n+1)\times M(n+1,n)) &= m(n+1)-(m-k)(n+1-k)+n(n+1) \notag \\
&= (m+n)(n+1)-(m-k)(n+1-k)\notag
\end{align}
\begin{align}
\dim (M_{k+1}(m,n+1)\times M(n+1,n)) &= m(n+1)-(m-k-1)(n-k)+n(n+1) \notag \\
&= (m+n)(n+1)-(m-k-1)(n-k)\notag
\end{align}
and
\begin{align}
\dim(\mu^{-1}(M_k(m,n))) &\ge m(n+1)+n(n+1)-(m-k)(n-k) \notag \\
&= (m+n)(n+1)-(m-k)(n-k)\notag,
\end{align}
by  the subadditivity of codimension for $\mu$ (see, for example, \cite[Theorem 17.24]{Harris}).

Moreover, we prove:

\begin{prop}
\label{prop:dimgen}
$
\dim\left(\mu^{-1}(M_k(m,n))\setminus (M(m,n+1)\times M_{n-1}(n+1,n))\right) = (m+n)(n+1)-(m-k)(n-k).
$
\end{prop}

\begin{proof}
 Let $A=(a_{ij})\in M(m,n+1)$ and $B=(b_{jk})\in M(n+1,n)$.  By definition $\mu(A,B)=AB$. Denote by $B^t$ the transpose of $B$. Working with coordinate order
 \[
(c_{11},c_{12},\ldots,c_{1n}|c_{21},c_{22},\ldots,c_{2n}|\ldots|c_{m1},c_{m2},\ldots,c_{mn})
\]
on $M(m,n)$ and
{\small
\[
(a_{11},a_{12},\ldots,a_{1,n+1}|a_{21},a_{22}\ldots,a_{2,n+1}|\ldots|a_{m1},a_{m2},\ldots,a_{m,n+1}|b_{11},\ldots|\ldots|\ldots b_{n+1,n})
\]}
on $M(m,n+1)\times M(n+1,n)$ the jacobian matrix of $\mu$ at $(A,B)$ is composed of two blocks:
 
\[
\left(\begin{array}{cccc|c}
 B^t & 0 & \dots & 0  & \dots  \\
 0 & B^t & \dots & 0  & \dots \\
 \vdots & \vdots&& \vdots&  \\
 0 & 0 & \dots & B^t  & \dots \\
 \end{array}\right)
 \]
 where the first block corresponding to $(\partial \mu/\partial a_{ij})$ has $m$ copies of $B^t$, and the second block is equivalent, via row permutations, with a similar matrix containing $n$ copies of $A$.
 
This matrix has  maximal rank $mn$ if $B$ has maximal rank $n$. Then $\mu$ is surjective  and smooth outside $M(m,n+1)\times M_{n-1}(n+1,n)$ and hence the fibres are equidimensional and $\mu^{-1}$ preserves codimension on the complement of this locus. 
  \end{proof}
  

Therefore we have:
{
\begin{eqnarray}\label{E:equal1}
&&\dim\left(\mu^{-1}(M_k(m,n))\setminus (M(m,n+1)\times M_{n-1}(n+1,n))\right) \\
\nonumber 
&& - \dim \left((M_k(m,n+1)\times M(n+1,n))\setminus (M(m,n+1)\times M_{n-1}(n+1,n))\right) \\
\nonumber
&& = m-k 
\end{eqnarray}
}
and
\begin{eqnarray}\label{E:equal2}
&&\dim \left((M_{k+1}(m,n+1)\times M(n+1,n))\setminus (M(m,n+1)\times M_{n-1}(n+1,n))\right)\\
\nonumber 
&&-\dim\left(\mu^{-1}(M_k(m,n))\setminus (M(m,n+1)\times M_{n-1}(n+1,n))\right) = n-k.
\end{eqnarray}



\section{The comparison of degeneracy loci} 
\label{sec:comparison}

Let $X$ be an arbitrary integral algebraic variety, and consider a diagram of vector bundles as at the beginning:
\[
\xymatrix{
0 \ar[r] & H \ar[r] & F' \ar[r]^-\pi & F \ar[r] & 0 \\
&&E \ar[u]^-\sigma\ar[ur]_-{\pi\sigma}}
\]
with $\rk(E)=m$, \ $\rk(F')=n+1$ and $\rk(F)=n$. The problem of comparing dimensions of degeneration loci is local, and hence we may assume that the three vector bundles are trivial and the morphisms are given by matrices. With the convention that a matrix defines a morphism by multiplication on the left with row vectors, this diagram induces a natural morphism:
\[
f: X \longrightarrow M(m,n+1) \times M(n+1,n),\ f=(f_1,f_2).
\]

Note that, by definition, the image of the second component $f_2$ is contained in $M(n+1,n)\setminus M_{n-1}(n+1,n)$.

We have the identifications:
\[
D_k(\sigma) = f^{-1}\left(M_k(m,n+1)\times M(n+1,n)\right),
\]
\[
D_{k+1}(\sigma) = f^{-1}\left(M_{k+1}(m,n+1)\times M(n+1,n)\right)
\]
and
\[
D_k(\pi\sigma) = f^{-1}(\mu^{-1}(M_k(m,n))).
\]

We prove:

\begin{prop}
\label{prop:dim}
Assume that
no irreducible component of $D_k(\pi\sigma)$ is contained in $D_k(\sigma)$.
If $D_k(\pi\sigma)$ is non empty, 
then
\begin{equation}
\label{E:equal3}
\dim(D_k(\pi\sigma))\ge \dim(D_{k+1}(\sigma))  - (n-k).
\end{equation}
If $D_k(\sigma)$ is non empty, 
then
\begin{equation}
\label{E:equal4}
\dim(D_k(\sigma)) \ge \dim(D_k(\pi\sigma)) - (m-k)
\end{equation}
and
\begin{equation}
\label{E:equal5}
\dim(D_k(\sigma)) \ge \dim(D_{k+1}(\sigma))  - (m+n-2k).
\end{equation}
\end{prop}

\begin{proof}
We prove (\ref{E:equal3}). Note that the hypothesis also implies that no irreducible component of $D_{k+1}(\sigma)$ is contained in $D_k(\sigma)$. We use the smoothness of $M_{k+1}(m,n+1)\setminus M_k(m,n+1)$ and the fact that $(M_{k+1}(m,n+1)\setminus M_k(m,n+1))\times M(n+1,n)$ intersects $\mu^{-1}(M_k(m,n))$. By the hypothesis, we know that
\[
\mathrm{dim}(D_{k+1}(\sigma))=\mathrm{dim}(D_{k+1}(\sigma)\setminus D_k(\sigma)).
\]
We apply Proposition \ref{prop:dimgen} and subbadditivity of codimension for the restriction of $f$ to $f^{-1}((M_{k+1}(m,n+1)\setminus M_k(m,n+1))\times M(n+1,n))=D_{k+1}(\sigma)\setminus D_k(\sigma)$ and the subvariety $D_k(\pi\sigma)\setminus D_k(\sigma)=f^{-1}\left(\mu^{-1}(M_k(m,n))\setminus (M_k(m,n+1)\times M(n+1,n))\right)$.

We prove (\ref{E:equal4}). First note that, since $D_{k-1}(\pi\sigma)\subset D_k(\sigma)\subset D_k(\pi\sigma)$, the hypothesis implies that no irreducible component of $D_k(\pi\sigma)$ is contained in  $D_{k-1}(\pi\sigma)$ and hence
\[
\mathrm{dim}(D_k(\pi\sigma))=\mathrm{dim}\left(D_k(\pi\sigma)\setminus D_{k-1}(\pi\sigma)\right).
\]

Since $M_k(m,n)\setminus M_{k-1}(m,n)$ is smooth, and $\mu$ is smooth on $(M(m,n+1)\times M(n+1,n))\setminus(M(m,n+1)\times M_{n-1}(n+1,n))$ (from the proof of Proposition \ref{prop:dimgen}), it follows that any $(A,B)\in \mu^{-1}(M_k(m,n)\setminus M_{k-1}(m,n))\setminus\left(M(m,n+1)\times M_{n-1}(n+1,n)\right)$ is a smooth point of $\mu^{-1}(M_k(m,n))$. Taking into account that the image of $f$ is in the complement of $M(m,n+1)\times M_{n-1}(n+1,n)$, (\ref{E:equal4}) follows from the subadditivity of codimension for the restricted map:
\[
f:D_k(\pi\sigma)\setminus D_{k-1}(\pi\sigma)\to \mu^{-1}(M_k(m,n))\setminus \mu^{-1}(M_{k-1}(m,n)).
\] 

The inequality (\ref{E:equal5}) follows from (\ref{E:equal3}) and (\ref{E:equal4}) by addition.

Note that if $D_k(\sigma)$ is non empty, then $D_k(\pi\sigma)$ is also non empty, and the non emptiness of $D_k(\pi\sigma)$ implies the non emptiness of  $D_{k+1}(\sigma)$.
%
%
%
%
%
\end{proof}

\begin{rmk}
Under some positivity assumptions, for example $E^*\otimes H$ be ample as in \cite{Fulton-Harris-Lazarsfeld}, the non emptiness of the corresponding degeneracy loci follows. In fact, \cite[Lemma 4]{Fulton-Harris-Lazarsfeld} reduces, by taking hyperplane sections, to proving the non emptiness of $D_k(\sigma)$ for the case $\mathrm{dim}(D_k(\pi\sigma)) = m-k$. Note that the proof of \cite[Lemma 4]{Fulton-Harris-Lazarsfeld} cannot be adapted to our case.
\end{rmk}

\begin{rmk}
The generic situation from section \ref{sec:generic} corresponds to the case $X=M(m,n+1) \times M(n+1,n)$ and $f=\mathrm{id}$ with trivial bundles $E$, $F$, $F'$ and naturally defined $\sigma$ and $\pi$.
\end{rmk}

\section{Secant loci}

Let $C$ be a smooth projective curve of genus $g$ and $n \ge 1$ be an integer. Denote by $\Xi_n \subset C\times C_n$ the universal divisor on the $n$--th symmetric product  $C_n$  of $C$. Consider the two projections $\pi:C\times C_n\to C$, respectively $\pi_n:C\times C_n\to C_n$. For any globally generated line bundle $L$ of degree $d$ on $C$ with $h^0(L)=r+1$, the \emph{secant bundle} of $L$ is the rank--$n$ vector bundle on  $C_n$ defined by: 
 \[
 E_{L,n} := \pi_{n*}(\pi^*L\otimes\O_{\Xi_n}).
 \]
 For any $\xi\in C_n$, the fibre of $E_{L,n}$ over $\xi$ is isomorphic to $L|_\xi$.
Note that $\pi_{n*}\pi^*L \cong H^0(L)\otimes\O_{C_n}$ and hence  we have a sheaf morphism
\[
e_{L,n}:H^0(L)\otimes\O_{C_n}\to E_{L,n}.
\]
The morphism $e_{L,n}$ is generically surjective for $n\le r$ since for a general effective divisor $\xi$ of degree $n$ on $C$ the map $H^0(L)\to L|_\xi$ is surjective.

For any $k\le n-1$, the \emph{secant locus} $V^k_n(L)$ is the closed subscheme $V^k_n(L) :=D_k(e_{L,n})\subset C_n$ \cite{Coppens-Martens}. If $L$ is very ample, it parametrizes the $n$--secant $(k-1)$--planes in the induced embedding. 
The secant loci have been recently used in connection with syzygy problems, \cite{Aprodu-Sernesi}, \cite{Farkas-Kemeni}.

The expected dimension of $V^k_n(L)$ is $n-(r+1-k)(n-k)$ and hence, if non--empty, then $V^k_n(L)$ has dimension $\ge n-(r+1-k)(n-k)$.

Consider $p\in C$ a general point that defines an embedding $C_n\cong p+C_n\subset C_{n+1}$. Since its pullback to the cartesian product $C^{n+1}$ is 
\[
(\{p\}\times C\times\ldots\times C)+(C\times \{p\}\times C\times\ldots\times C)+\ldots+(C\times C\times\ldots\times C\times \{p\}),
\] 
it follows that $C_n$ is moreover an ample divisor (see also \cite[Lemma 2.7]{Fulton-Lazarsfeld} for another proof). 

For any $n$, we have a short exact sequence of vector bundles on $C_n$:
\[
0\to \mathcal O_{C_{n+1}}(-C_n)|_{C_n}\to E_{L,n+1}|_{C_n}\to E_{L,n}\to 0.
\]
Indeed, the kernel of the surjective morphism $E_{L,n+1}|_{C_n}\to E_{L,n}$ is a line bundle on $C_n$, and hence it is isomorphic to $\mathrm{det}(E_{L,n+1}|_{C_n})\otimes \mathrm{det}(E_{L,n})^{-1}$. Using the isomorphisms $\mathrm{det}(E_{L(p),n+1})\cong \mathrm{det}(E_{L,n+1})\otimes\mathcal O_{C_{n+1}}(C_n)$ and $\mathrm{det}(E_{L(p),n+1})|_{C_n}\cong \mathrm{det}(E_{L,n})$ (see, for example \cite[5.2.3, p. 71]{Aprodu-Nagel}) the claim follows.

\medskip

We apply the result from the previous section to $X=C_n$, $H=\mathcal O_{C_{n+1}}(-C_n)|_{C_n}$, $F'=E_{L,n+1}|_{C_n}$, $F=E_{L,n}$, and $E=H^0(L)\otimes \mathcal O_{C_n}$, where $\sigma$ is the evaluation map. Note that $D_k(\pi\sigma)=V_n^k(L)$, $D_k(\sigma)=V_{n+1}^k(L)\cap C_n$, $D_{k+1}(\sigma)=V_{n+1}^{k+1}(L)\cap C_n$ and hence $\dim(D_k(\sigma))=\dim(V_{n+1}^k(L))-1$ and $\dim(D_{k+1}(\sigma))=\dim(V_{n+1}^{k+1}(L))-1$ by the genericity of $p$. 
Assuming non emptiness for the suitable secant loci, the inequalities (\ref{E:equal3}), (\ref{E:equal4}) and (\ref{E:equal5}) yield to the following excess dimension result:

\begin{thm}
\label{thm:main}
If $V_n^k(L)\ne\emptyset$ then
\begin{equation}\label{E:equal6}
\dim\left(V_n^k(L)\right)\ge \dim\left(V_{n+1}^{k+1}(L)\right) - (n-k+1).
\end{equation}
If $V_{n+1}^k(L)\ne\emptyset$ and moreover $\mathrm{dim}(V_{n+1}^k(L))\ge 1$ then
\begin{equation}\label{E:equal7}
\dim\left(V_{n+1}^k(L)\right)\ge \dim\left(V_n^k(L)\right) - (r-k)
\end{equation}
and
\begin{equation}\label{E:equal8}
\dim\left(V_{n+1}^k(L)\right)\ge \dim\left(V_{n+1}^{k+1}(L)\right) - (r+n-2k+1).
\end{equation}
\end{thm}


\begin{proof}
We only need to verify that no irreducible component of $V^k_n(L)$ is contained in $V^k_{n+1}(L)\cap C_n$. To this end, we use the genericity of $p$. Note that all the divisors in $V^k_{n+1}(L)\cap C_n$ contain $p$ in the support.  Let $V\subset V^k_{n}(L)$ be an irreducible component and $D_V\subset C$ be the intersection of the supports of divisors $D$ in $V$. If we pick $p$ outside $D_V$, then it is clear that $V\not\subset V^k_{n+1}(L)\cap C_n$. Choosing $p$ in the complement of the union of these loci $D_V$, we obtain the result.
\end{proof}

\begin{rmk}
Since the expected dimension of $V_n^k(L)$ is $n-(r+1-k)(n-k)$ and the expected dimension of $V_{n+1}^{k+1}(L)$ is $(n+1)-(r-k)(n-k)$ we note that if the dimension of $V_n^k(L)$ equals the expected dimension, then the same is true for $V_{n+1}^{k+1}(L)$.
\end{rmk}

\begin{rmk}
In the special case $L=K_C$, we have an identification $V_n^k(K_C)=C_n^{n-k}$ and hence
\[
\dim\left(V_n^k(K_C)\right)=\dim\left(W_n^{n-k}(C)\right)+(n-k).
\]
Theorem 1 in \cite{Fulton-Harris-Lazarsfeld} corresponds to (\ref{E:equal6}). Corollary 2 in \cite{Fulton-Harris-Lazarsfeld} corresponds to (\ref{E:equal7}) and Corollary 3 in \cite{Fulton-Harris-Lazarsfeld} is~(\ref{E:equal8}). Note, however, that in \cite{Fulton-Harris-Lazarsfeld} no non-emptiness assumption is needed.
\end{rmk}

\begin{rmk}
\label{rmk:Jacobian}
If $h^1(L)=h$, the image of $V^k_n(L)$ in the Jacobian via the Abel-Jacobi map is the intersection $W_n(C)\cap\left((L-K_C)+W^{n-k-1+h}_{2g-2-d+n}(C)\right)$.
The proof follows from Riemann-Roch applied to $L$ and $L(-D)$ with $D\in V_n^k(L)$. 

In particular, $V_n^k(L)\ne \emptyset$ if and only if (compare with \cite{Farkas-Kemeni})
\[
L-K_C\in W_n(C)-W^{n-k-1+h}_{2g-2-d+n}(C)\subset \mathrm{Pic}_{d-2g+2}(C).
\] 
If $L$ is non special, then $V_n^k(L)\ne \emptyset$ if and only if $L-K_C\in W_n(C)-W^{n-k-1}_{2g-2-d+n}(C)\subset \mathrm{Pic}_{d-2g+2}(C)$. Note that, having fixed $d$, $n$ and $k$, the locus $W_n(C)-W^{n-k-1+h}_{2g-2-d+n}(C)$ decreases when $h$ increases, and hence the chances for $V_n^k(L)$ to be non empty also decrease.

In some cases, for instance, if $n\ge g$ or if $L-K_C\ge 0$, we have an inclusion $(L-K_C)+W^{n-k-1+h}_{2g-2-d+n}(C)\subset W_n(C)$, however,  this condition is not verified in general. In principle, by restriction to $W_n(C)$, the description from \cite{Fulton-Harris-Lazarsfeld} applies to obtain excess dimension results for the images of secant loci in $\mathrm{Pic}_n(C)$. If the fibres of the Abel-Jacobi maps over these loci are controllable (for example, if $n\le \mathrm{gon}(C)-2$), then one can pass from the Jacobian to the symmetric products and, assuming moreover that $h\le 1$, Theorem \ref{thm:main} (\ref{E:equal6}) can be improved in the sense that one can drop the non-emptiness assumption in the hypothesis. 
\end{rmk}

\begin{rmk}
Applying Theorem \ref{thm:main} we can simplify the statement of Proposition 2.6 in \cite{Aprodu-Sernesi} on condition ($\Delta_q$) and obtain a perfect analogue of the canonical case (\cite[Proposition 3.6]{Aprodu-Sernesi}). More precisely, using the terminology and the notation of loc.cit., if the dimension of the locus $V^{r-q}_{r-q+1}(L)$ equals the expected dimension $r-2q$ and $\mathrm{dim}(V^{r-q+1}_{r-q+3}(L))\le r-2q+1$ then the condition ($\Delta_q$) holds in the strong sense.
\end{rmk}

\begin{rmk}
A similar argument as in the proof of \cite[Lemma 2.2]{Aprodu-Sernesi} shows that no irreducible component of $V^k_n(L)$ is contained in $V^{k-1}_n(L)$ if $L$ is very ample and~$k\le r$.
\end{rmk}

\begin{rmk}
In \cite{Kemeny}, M. Kemeny applies our result to syzygies of curves.
\end{rmk}


\begin{thebibliography}{99999}
\bibitem{Aprodu-Nagel}
M. Aprodu, J. Nagel. 
Koszul Cohomology and Algebraic Geometry. 
\emph{University Lect. Series} 52, AMS, 2010.

\bibitem{Aprodu-Sernesi} 
M. Aprodu, E. Sernesi.  
Secant spaces and syzygies of special line bundles on curves. 
\emph{Alg. \& Number Theory} 9:3 (2015) 585--600.

\bibitem{Arbarello-Cornalba-Griffiths-Harris} 
E. Arbarello, M. Cornalba, P. Griffiths and J. Harris. 
Geometry of algebraic curves. vol. I. 
Springer--Verlag, 1984.


\bibitem{Coppens-Martens}
M. Coppens, G. Martens. 
Secant spaces and Clifford's theorem.
\emph{Compositio Math.}, 78:2 (1991) 193--212.

\bibitem{Farkas-Kemeni} 
G. Farkas, M. Kemeny. 
The generic Green-Lazarsfeld secant conjecture. 
\emph{Inventiones Math.}, 203:1 (2016) 265--301.

\bibitem{Fulton-Harris-Lazarsfeld} 
W. Fulton, J. Harris, R. Lazarsfeld. 
Excess linear series on an algebraic curve. 
\emph{Proc. American Math. Soc.} 92:3 (1984) 320--322.

\bibitem{Fulton-Lazarsfeld}
W. Fulton, R. Lazarsfeld.
On the connectedness of degeneracy loci and special divisors. 
\emph{Acta Math.} 146:1 (1981) 271--283.

\bibitem{Harris} 
J. Harris. 
Algebraic Geometry. A first Course. 
Springer--Verlag, 1992.

\bibitem{Kemeny}
M. Kemeny. 
The extremal secant conjecture for curve of arbitrary gonality. 
preprint arXiv:1512:00212.

\bibitem{Kempf} 
G. Kempf. 
Images of homogeneous vector bundles and varieties of complexes. 
\emph{Bull. American Math. Soc.} 81:5 (1975) 900--901.

\end{thebibliography}
\end{document}